\documentclass[12pt,a4paper]{article}
\usepackage{amssymb, amsmath, amscd, euscript, amsthm}

\usepackage[utf8]{inputenc}
\usepackage[english]{babel}

\DeclareMathOperator\lcm{lcm}
\newtheorem{theorem}{Theorem}
\newtheorem{lemma}{Lemma}
\author{F. Petrov}

\title{Asymptotics of Landau--Okhotin function}
\begin{document}
\maketitle
\begin{abstract}
Landau function $g(n)$ \cite{L} 
is the maximal possible least common multiple
of several positive integers with sum not exceeding $n$.
Under additional assumptions that these
numbers are the differences of disjoint
arithmetic progressions the maximum is denoted 
$\tilde{g}(n)$, it was introduced by Okhotin \cite{O}.
We find a sharp logarithmic asymptotics of $\tilde{g}(n)$.
\end{abstract}

First of all, we introduce some notations:

$X\sim Y$ means $X/Y\to 1$;

$X\preccurlyeq Y$ means $X\leqslant (1+o(1))Y$;

$X\succcurlyeq Y$ means $X\geqslant (1+o(1))Y$.

$\nu_p(m)$ denotes the exponent of a prime $p$
in the prime factorization of $m$.

\smallskip

Let $n$ be a positive integer. Consider
the arrays  $(\pi_1,\ldots,\pi_k)$ of positive integers
such that

(i) $\pi_1+\ldots+\pi_k\leqslant n$;

(ii) there exist disjoint arithmetic progressions
of the form $f_i+\pi_i\mathbb{Z}$,
$f_i\in \mathbb{Z}$,
$i=1,\ldots,k$. 

The maximal possible value
of $\lcm(\pi_1,\ldots,\pi_k)$ under conditions
 (i-ii) is called \emph{Landau--Okhotin function} $\tilde{g}(n)$.

Landau function $g(n)$ \cite{L}
is defined analogously but without condition
(ii). $g(n)$ attracts a lot of interest as
it is the maximal order of an element
of the symmetric group $S_n$.

Function $\tilde{g}$ was studied by A. Okhotin 
\cite{O} in relation with finite automata theory.

In \cite{O} it is proved that \begin{equation}\label{1}
\frac29n\log^{2} n\preccurlyeq \log^3 \tilde{g}(n)\preccurlyeq
2n\log^{2} n\end{equation}

Note that
$\log g(n)\sim \sqrt{n\log n}$ \cite{L,M,S}, i.e. 
the condition (ii)
changes asymptotics dramatically. On the informal
level the explanation is that for satisfying (ii),
the numbers $\pi_i$ must have large common divisors,
which decreases their least common multiple.

Our goal is to improve \eqref{1}
by proving the sharp logarithmic asymptotics of
$\tilde{g}$:

\begin{theorem}
\begin{equation}\label{2} 
\log^3 \tilde{g}(n)\sim \frac14n\log^{2} n.\end{equation}
\end{theorem}

Theorem 1 answers a question posed in \cite{O}.

\smallskip

We use many times the following simple

\begin{lemma}
If $X,Y$ go to $+\infty$ and 
$X\sim c Y^\alpha \log^\theta Y$ for $x,\alpha>0$ and $\theta\in \mathbb{R}$,
then $Y^\alpha \sim c^{-1} \alpha^\theta X/\log^\theta X$.
\end{lemma}

\begin{proof}
If $X\sim Z$, then $\log Z=\log X+o(1)$, and in particular 
$\log Z\sim \log X$. Thus $$\log X\sim \log c+\alpha \log Y+\theta
\log \log Y\sim \alpha \log Y,$$
and 
$$
Y^\alpha\sim c^{-1} X\log^{-\theta} Y\sim c^{-1} \alpha^\theta \cdot X\cdot \log^{-\theta} X.
$$
\end{proof}

Also the following equivalent forms of Prime Numbers Theorem
are used. If $p_1<p_2<\ldots$ is a sequence of primes, then

(A) $p_k\sim k \log k$;

(B) $p_1^\theta+\ldots+p_k^\theta\sim \frac1{\theta+1} k^{\theta+1}\log^\theta k$
for all $\theta>-1$;

(C) $\log p_1+\ldots+\log p_k\sim k\log k$;

(D) $\log \lcm(1,2,\ldots,k)\sim k$.

\begin{proof}[Proof of Theorem 1]
We start with a lower bound of $\tilde{g}(n)$. 

Fix a large positive integer $s$.
Denote $M:=\sqrt{1}+\sqrt{2}+\ldots+\sqrt{s}$,
$r_i:=\lceil M i^{-1/2}\rceil$ for
$i=1,2,\ldots,s$. 
Then $\sum 1/r_i\leqslant \sum \sqrt{i}/M\leqslant 1$. 
Let $N$ be a positive integer divisible by 
 $\lcm(r_1,\ldots,r_s)$. Consider $k:=sN$ numbers $\theta_1,\ldots,\theta_{sN}$:
$\theta_1=\ldots=\theta_N=r_1N$, $\theta_{N+1}=\ldots=\theta_{2N}=r_2N$ 
and so on. 

Our current goal is to build $sN$ disjoint arithmetic
progressions with differences 
$\theta_1,\ldots,\theta_{sN}$. 
For that, choose $s$ disjoint sets of residues
modulo $N$, where the $i$-th set  ($i=1,\ldots,s$) 
contains $N/r_i$ residues. This is possible
since $N/r_i$ are positive integers whose sum does not
exceed $N$. Every residue modulo $N$ naturally
corresponds to an infinite progression with difference 
$N$. For every residue of the $i$-th set partition
this progression onto $r_i$ progressions
with difference  $r_iN$. Totally we get 
$r_i\cdot (N/r_i)=N$ 
arithmetic progressions 
with difference $r_iN$ for every $i=1,\ldots,s$, 
and these progressions are disjoint. This is what we need.

Let now $p_1,\ldots,p_{k}$ be the first
$k=sN$ prime numbers in the increasing order. 
Let the number 
$N$ divisible by $\lcm(r_1,\ldots,r_s)$
be maximal possible for which 
$\theta_1 p_1+\ldots+\theta_k p_{k}\leqslant n$.

Denote $\pi_i=\theta_ip_i$. 
Conditions (i) and (ii) are enjoyed, and for the least
common multiple we have
$$
\log \lcm(\pi_1,\ldots,\pi_k)\geqslant \log p_1\ldots p_k \sim k \log k,
$$
thus the lower bound of $\tilde{g}(n)$ reduces
to a lower bound of $k$. Denoting
$P_m=p_1+\ldots+p_m\sim \frac12 m^2\log m$, we have  
\begin{align*}
\sum \theta_i p_i\sim N\left(r_1P_N+r_2(P_{2N}-P_N)+\ldots+r_s(P_{sN}-P_{(s-1)N})\right)\\ \sim
\frac12 N^3 \log N \left(r_1+3r_2+\ldots+(2s-1)r_s\right).
\end{align*}

From maximality of $N$ we get $N^{3+o(1)}=n$, $\log N\sim \frac13\log n$,
and
\begin{align*}
N^3\sim \frac{6n}{\log n(r_1+3r_2+\ldots(2s-1)r_s)},\\ (k\log k)^3\sim \frac1{27}s^{3}N^3\log^3 n\sim \frac29 \kappa(s) n\log^2n,\quad \kappa(s):=\frac{s^3}{r_1+3r_2+\ldots+(2s-1)r_s}.
\end{align*}

Thus for proving the required
bound $\log^3 \tilde{g}(n)\succcurlyeq \frac14 n\log^2 n$ 
it remains to prove that the constant
 $\kappa(s)$ may be arbitrarily close to
 $9/8$. For large $s$ we have
$$
\kappa(s)\sim \frac{s^3}{\sum_{i=1}^s 2iMi^{-1/2}}\sim \frac{s^3}{2M^2}\to 9/8,
$$
since $M\sim \int_0^s \sqrt{x}dx=\frac23 s^{3/2}$. 
The lower bound is proved.
\medskip

Now the upper bound. Let $\pi_1,\ldots,\pi_k$ and $f_1,\ldots,f_k$ 
satisfy conditions (i-ii). 
Put $\pi_i=q_i w_i$, where 
$w_i=\gcd(\pi_i,\lcm\{\pi_j:j\ne i\})$. If $q_i=1$, 
we may simply remove the element $\pi_i$ 
(conditions (i-ii) are still satisfied and the
least common multiple is not changed), 
so let us suppose that this does not happen.

Note that \begin{equation}\label{lcm-above}\lcm(\pi_1,\ldots,\pi_k)= 
q_1\ldots q_k \lcm(w_1,\ldots w_k).\end{equation}

Indeed, take an arbitrary prime
 $p$, denote $\nu_p(\pi_i)=\alpha_i$.
Without loss of generality we have $\alpha_1\geqslant \ldots \geqslant \alpha_k$. 
Then $\nu_p(\lcm(\pi_1,\ldots,\pi_k))=\alpha_1$,
$\nu_p(q_1)=\alpha_1-\alpha_2$, $\nu_p(q_i)=0$ for $i>1$,
$\nu_p(w_i)=\min(\alpha_2,\alpha_i)$ for $i=1,\ldots,k$,
therefore $\nu_p(\lcm(w_1,\ldots w_k))=\alpha_2$.
Thus $\nu_p$ of both sides of \eqref{lcm-above}
equals $\alpha_1$.

We proceed with bounding  $q_1\ldots q_k$.

The numbers $f_i$ and $f_j$ are different
modulo
$\gcd(\pi_i,\pi_j)=\gcd(w_i,w_j)$,
thus the arithmetic progressions
$f_i+w_i\mathbb{Z}$ are disjoint, and the sum of their densities
(in the set of integers) does not exceed 1:
$1/w_1+\ldots+1/w_k\leqslant 1$. 

Since $q_i$'s are coprime, they are not
less than $k$ first prime numbers.
Applying, like in \cite{O}, the Cauchy--Bunyakovskiy--Schwarz 
inequality and the inequality between arithmetic and geometric means,
we get
\begin{align*}
n\geqslant \left(\sum q_iw_i\right)\left(\sum 1/w_i\right)\geqslant \left(\sum \sqrt{q_i}\right)^2\succcurlyeq \left(\frac23k^{3/2}\log^{1/2} k\right)^2,\\
k^3\preccurlyeq \frac{27}4\frac{n}{\log n}, \\
q_1\ldots q_k\leqslant \left(\frac{\sqrt{q_1}+\ldots+\sqrt{q_k}}k\right)^{2k}
\leqslant \left(\frac{n}{k^2}\right)^k,\\
\log^3 q_1\ldots q_k\leqslant k^3\log^3 \frac{n}{k^2}\preccurlyeq \frac14n\log^2 n
\end{align*}
(we used that the function  
$k\log(nk^{-2})$ is increasing on the 
segment $[1,n^{1/3}]$).

Now we bound the factor 
$\lcm(w_1,\ldots,w_k)$ in \eqref{lcm-above}.

Let $w_1,\ldots,w_l$ exceed  $n^{1/3}$, and 
$w_{l+1},\ldots,w_k$ be at most $n^{1/3}$. 
the logarithm of the least common multiple of
 $w_{l+1},\ldots,w_k$ does not exceed the logarithm of
 the least common multiple of all numbers from 1 to $n^{1/3}$, which by 
 (D) is
$O(n^{1/3})=o(n^{1/3}\log^{2/3} n)$. 

For bounding the least common multiple of
$w_1,\ldots,w_l$ we start with bounding $l$:
$$
n\geqslant q_1w_1+\ldots+q_lw_l\geqslant 
n^{1/3} (q_1+\ldots+q_l)\geqslant c_0 n^{1/3}l^2\log l 
$$
for certain universal constant $c_0>0$.

Thus
$$
\log \lcm(w_1,\ldots,w_l)\leqslant 
\log (w_1\ldots w_l)\leqslant
l\log n=O(n^{1/3}\log^{1/2} n)=o(n^{1/3}\log^{2/3} n).
$$

Therefore $$\log \lcm(w_1,\ldots,w_k)
\leqslant \log \lcm(w_1,\ldots,w_l)+
\log \lcm(w_{l+1},\ldots,w_k)=o(n^{1/3}\log^{2/3} n),$$
as needed.
\end{proof}

I am grateful to A. Okhotin for attention to the work.

\end{document}